\documentclass[10pt]{amsart}
\usepackage{amssymb}
\usepackage{amscd}
\theoremstyle{plain}
\newtheorem{theorem}{Theorem}
\newtheorem*{mainthm}{Main Theorem}
\theoremstyle{definition}

\newtheorem{proposition}{Proposition}[section]

\newtheorem{corollary}[proposition]{Corollary}
\newtheorem{lemma}[proposition]{Lemma}
\newtheorem{definition}[proposition]{Definition}
\theoremstyle{remark}

\DeclareMathOperator{\stat}{stat}

\DeclareMathOperator{\nacc}{nacc}
\DeclareMathOperator{\Sk}{Sk}

\DeclareMathOperator{\otp}{otp}
\DeclareMathOperator{\Ch}{Ch}

\DeclareMathOperator{\ran}{ran}

\DeclareMathOperator{\cf}{cf}

\DeclareMathOperator{\tr}{Tr}

\newcommand{\sk}{\vskip.05in}

\newcommand{\restr}{\upharpoonright}

\newcommand{\subs}{\subseteq}

\numberwithin{equation}{section}
\begin{document}
\title{Getting more colors}
\author{Todd Eisworth}
\address{Department of Mathematics\\
         Ohio University\\
         Athens, OH 45701}
\email{eisworth@math.ohiou.edu}
\date{\today}
\begin{abstract}
We establish a coloring theorem for successors of a singular cardinals, and use it prove that for any such cardinal $\mu$, we have $\mu^+\nrightarrow[\mu^+]^2_{\mu^+}$ if and only if $\mu^+\nrightarrow[\mu^+]^2_{\theta}$ for arbitrarily large $\theta<\mu$.
\end{abstract}

\maketitle

\section{Introduction}

Our aim in this note is to prove a type of ``negative stepping-up theorem'' for square-brackets partition relations at
successors of singular cardinals.  In order to state our results precisely, we need to recall the following bit of notation due originally to Erd\H{o}s, Hajnal, and Rado~\cite{EHR}:

\begin{definition}
 $\kappa\rightarrow [\lambda]^\mu_\theta$  means that for any function $F:[\kappa]^{\mu}\rightarrow\theta$,
(to which we refer as a {\em coloring}) we can find a set $H\subs\kappa$ of cardinality $\lambda$ for which
\begin{equation*}
\ran(F\restr [H]^\mu)\subsetneqq\theta.
\end{equation*}
\end{definition}

The negation of a square-brackets partition relation asserts the existence of a coloring which exhibits complicated behavior on every large subset of the domain. We will be concerned with relations of the form $\kappa\nrightarrow[\kappa]^2_\theta$, which states that one may color the pairs of ordinals from $\kappa$ with $\theta$ colors in such a way that $f\restr[A]^2$ assumes all colors for any set $A\in[\kappa]^\kappa$. (We will usually identify $[\kappa]^2$ with those pairs $\langle\alpha,\beta\rangle\in\kappa\times\kappa$ with $\alpha<\beta$.) We refer the reader to Chapter~XI of~\cite{ehmr} for a more comprehensive introduction to square-brackets partition relations and their negations.

We mentioned in the opening sentence that we aim to prove a sort of negative stepping-up theorem.  The terminology ``negative stepping-up theorem'' usually refers to results which increase the cardinal appearing on the left side of a given negative partition relation. This is not quite what we are after --- we assume the existence of certain colorings on a cardinal~$\kappa$ and prove that the number of colors can automatically be upgraded while keeping the ``domain'' $\kappa$ fixed.  The following simple proposition provides our motivation:

\begin{proposition}
\label{prop1}
The following two statements are equivalent for a cardinal $\mu$:
\begin{enumerate}
\sk
\item $\mu^+\nrightarrow [\mu^+]^2_{\mu^+}$
\sk
\item $\mu^+\nrightarrow [\mu^+]^2_{\mu}$
\sk
\end{enumerate}
\end{proposition}
\begin{proof}
It is clear that (1) implies (2), so assume we have a function $c:[\mu^+]^2\rightarrow\mu$ witnessing that $\mu^+\nrightarrow [\mu^+]^2_{\mu}$.  For each $\beta<\mu^+$, fix a function $g_\beta$ mapping $\mu$ onto $\beta$, and define
\begin{equation}
c^*(\alpha,\beta)= g_\beta(c(\alpha,\beta)).
\end{equation}
We will show that $c^*$ serves as a witness for $\mu^+\nrightarrow [\mu^+]^2_{\mu^+}$.

To see this, suppose $A\subs\mu^+$ is of size $\mu^+$, and let $\varsigma<\mu^+$ be arbitrary. Our goal is to produce $\alpha<\beta$ in $A$ for which $c^*(\alpha,\beta)=\varsigma$, so without loss of generality we may assume that $\varsigma<\min(A)$.

Given $\epsilon<\mu$, define
\begin{equation}
A_\epsilon:=\{\beta\in A: g_\beta(\epsilon)=\varsigma\}.
\end{equation}
Since $\varsigma<\min(A)$, it follows that $A=\bigcup_{\epsilon<\mu}A_\epsilon$. In particular, we can choose a single
$\epsilon<\mu$ for which $A_\epsilon$ has size $\mu^+$.  It follows that we can find $\alpha<\beta$ in $A_\epsilon$ (hence in~$A$) for which $c(\alpha,\beta)=\epsilon$, and so
\begin{equation}
c^*(\alpha,\beta)=g_\beta(c(\alpha,\beta))=g_\beta(\epsilon)=\varsigma,
\end{equation}
where the last equality holds because $\beta\in A_\epsilon$.
\end{proof}

This simple little argument applies in many other situations. For example, one easily obtains by the same method
the equivalence of $\mu^+\nrightarrow[\mu^+]^n_{\mu^+}$ and $\mu^+\nrightarrow[\mu^+]^n_{\mu}$ for any finite~$n$.

These results, however, are of no interest in the case where $\mu$ is a regular cardinal, as a celebrated result of Todor{\v{c}}evi{\'c}~\cite{acta} establishes  $\mu^+\nrightarrow[\mu^+]^2_{\mu^+}$ always holds when $\mu$ is a regular cardinal. On the other hand, the case where $\mu$ is singular is a much different story because it is still unknown whether $\mu^+\nrightarrow[\mu^+]^2_{\mu^+}$ (or even the much weaker $\mu^+\nrightarrow[\mu^+]^{<\omega}_{\mu^+}$) must hold.
If we assume that $\mu$ is singular, then there is a natural way to attempt to strengthen Proposition~\ref{prop1}  --- one may ask if $\mu^+\rightarrow[\mu^+]^2_{\mu^+}$ follows only by assuming that $\mu^+\nrightarrow[\mu^+]^2_\theta$ holds for arbitrarily large $\theta<\mu$. Experience suggests that the answer should be yes, and that the result should follow by one of the standard ``patching arguments'' common in this area of set theory.  Unfortunately, a naive attempt at this yields only the following weak result:

\begin{proposition}
\label{weakprop}
Suppose $\mu$ is a singular cardinal.  If $\mu^+\nrightarrow [\mu^+]^2_\theta$ for arbitrarily large $\theta<\mu$, then
\begin{equation*}
\mu^+\nrightarrow[\mu^+]^4_{\mu^+}.
\end{equation*}
\end{proposition}
\begin{proof}
We give only a sketch. Note the same argument given in Proposition~1 tells us it suffices to establish
\begin{equation*}
 \mu^+\nrightarrow[\mu^+]^4_\mu.
\end{equation*}

Let $\langle\theta_i:i<\cf(\mu)\rangle$ be an increasing sequence of cardinals unbounded in~$\mu$.  For each $i<\cf(\mu)$, let $d_i$ be a coloring witnessing $\mu^+\nrightarrow[\mu^+]^2_{\theta_i}$ (note that our assumptions imply such colorings exist for {\em every} $\theta<\mu$). Now $\mu^+\nrightarrow[\mu^+]^2_{\cf(\mu)}$ (a result of Shelah --- see Conclusion~4.1 on page~67 of~\cite{cardarith}), so we can fix a coloring $c$ witnessing this.

We now use $c$ to patch together the colorings $d_i$, that is, we define a function $f$ on $[\mu^+]^4$ by setting
\begin{equation}
f(\alpha,\beta,\gamma,\delta)= d_{c(\gamma,\delta)}(\alpha,\beta).
\end{equation}

Given $\varsigma<\mu$ and unbounded $A\subs \mu^+$, we must find $\alpha<\beta<\gamma<\delta$ in $A$ for which
\begin{equation*}
f(\alpha,\beta,\gamma,\delta)=\varsigma.
\end{equation*}
This is, however, quite straightforward and the result follows.
\end{proof}

The rest of this paper is essentially concerned with turning the ``4'' in Proposition~\ref{weakprop} into a~``2''. We accomplish this by proving the following theorem (in {\sf ZFC}) from which the desired result can be deduced as an easy corollary.

\begin{mainthm}
Suppose $\mu$ is a singular cardinal. There is a function
 \begin{equation}
 D:[\mu^+]^2\rightarrow [\mu^+]^2\times\cf(\mu)
 \end{equation}
 such that for any unbounded $A\subs\mu^+$, there is a stationary $S\subs\mu^+$ such that
 \begin{equation}
 [S]^2\times\cf(\mu)\subs\ran(D\restr [A]^2).
 \end{equation}
\end{mainthm}

\section{Background material}
\noindent{\sf Minimal Walks}

\smallskip

Recall that $\bar{e}=\langle e_\alpha:\alpha<\lambda\rangle$ is a $C$-sequence for the cardinal
$\lambda$ if $e_\alpha$ is closed unbounded in $\alpha$ for each $\alpha<\lambda$.
Following Todor{\v{c}}evi{\'c}, given $\alpha<\beta<\lambda$ the {\em minimal walk from $\beta$ to $\alpha$ along $\bar{e}$}
is defined to be the sequence $\beta=\beta_0>\dots>\beta_{n}=\alpha$ obtained by setting
\begin{equation}
\beta_{i+1}=\min(e_{\beta_i}\setminus\alpha).
\end{equation}
The {\em trace} of the walk from $\beta$ to $\alpha$ is defined by
\begin{equation}
\tr(\alpha,\beta)=\{\beta=\beta_0>\beta_1>\dots>\beta_n=\alpha\},
\end{equation}
that is, $\tr(\alpha,\beta)$ is the set of all ordinals visited on the walk from $\beta$ down to $\alpha$ along~$\bar{e}$.

There are other standard parameters associated with minimal walks that are quite important for our purposes.  For example, we need the function $\rho_2:[\lambda]^2\rightarrow\omega$ giving the length of the walk from $\beta$ to $\alpha$, that is,
\begin{equation}
\rho_2(\alpha,\beta)=\text{ least $i$ for which $\beta_i(\alpha,\beta)=\alpha$}.
\end{equation}
For $i\leq\rho_2(\alpha,\beta)$, we set
\begin{equation*}
\beta_i^-(\alpha,\beta)=
\begin{cases}
0 &\text{if $i=0$},\\
\sup(e_{\beta_j(\alpha,\beta)}\cap\alpha) &\text{if $i=j+1$ for $j<\rho_2(\alpha,\beta)$}.
\end{cases}
\end{equation*}
Note that if $0<i<\rho_2(\alpha,\beta)$ then
\begin{itemize}
\item  $\beta_i^-(\alpha,\beta)=\max(e_{\beta_{i-1}(\alpha,\beta)}\cap\alpha)$ (as opposed to ``sup''),
\sk
\item $\beta_i^-(\alpha,\beta)<\alpha<\beta_i(\alpha,\beta)$, and
\sk
\item $\beta_i(\alpha,\beta)=\min(e_{\beta_{i-1}(\alpha,\beta)}\setminus \beta_i^-(\alpha,\beta)+1)$.
\sk
\end{itemize}
Thus, for $0<i<\rho_2(\alpha,\beta)$, the ordinals $\beta^-_i(\alpha,\beta)$ and $\beta_i(\alpha,\beta)$ are the two consecutive elements in $e_{\beta_{i-1}(\alpha,\beta)}$ which ``bracket''~$\alpha$.

For our purposes, we need to analyze what happens in the case where $i=\rho_2(\alpha,\beta)$. In this situation,  we have
\begin{itemize}
\sk
\item $\beta_{\rho_2(\alpha,\beta)}^-(\alpha,\beta)\leq\alpha=\beta_{\rho_2(\alpha,\beta)}(\alpha,\beta)$, and
\sk
\item $\beta^-_{\rho_2(\alpha,\beta)}(\alpha,\beta)<\alpha$ if and only if $\alpha\in\nacc(e_{\beta_{\rho_2(\alpha,\beta)-1}(\alpha,\beta)})$.
\end{itemize}
Notice that $\alpha$ must be an element of  $e_{\beta_{\rho_2(\alpha,\beta)-1}(\alpha,\beta)}$ by definition, and  $\beta^-_{\rho_2(\alpha,\beta)}(\alpha,\beta)$ is less than $\alpha$ precisely when $\alpha$ fails to be an accumulation point of $e_{\beta_{\rho_2(\alpha,\beta)-1}(\alpha,\beta)}$.

We are going to make use of some standard patterns of argument using minimal walks, and this is going to require a couple of more bits of notation. To wit, we define
\begin{gather}
\gamma(\alpha,\beta)=\beta_{\rho_2(\alpha,\beta)-1}(\alpha,\beta),\\
\gamma^-(\alpha,\beta)=\max\{\beta_i^-(\alpha,\beta):i<\rho_2(\alpha,\beta)\},
\intertext{and}
\label{etadef}
\eta(\alpha,\beta)=\max\{\beta_i^-(\alpha,\beta):i\leq\rho_2(\alpha,\beta)\}.
\end{gather}
The following proposition captures some standard facts about minimal walks; the proof is an easy induction.

\begin{proposition}
\label{minimalprop}
Suppose $\alpha<\beta$.
\begin{enumerate}
\item $\gamma^-(\alpha,\beta)<\alpha$, and if $\gamma^-(\alpha,\beta)<\alpha^*\leq\alpha$ then
\begin{equation}
\beta_i(\alpha,\beta)=\beta_i(\alpha^*,\beta)\text{ for $i<\rho_2(\alpha,\beta)$}.
\end{equation}
\sk
\item $\eta(\alpha,\beta)\leq\alpha$, and if it happens that $\eta(\alpha,\beta)<\alpha^*\leq\alpha$, then
\begin{equation}
\beta_i(\alpha,\beta)=\beta_i(\alpha^*,\beta)\text{ for $i\leq\rho_2(\alpha,\beta)$}.
\end{equation}
In particular,
\begin{equation}
\beta_{\rho_2(\alpha,\beta)}(\alpha^*,\beta)=\alpha.
\end{equation}
\end{enumerate}
\end{proposition}

Note that part (2) of the above proposition is of no interest unless we can guarantee $\eta(\alpha,\beta)<\alpha$ (or equivalently, guarantee $\alpha\in\nacc(e_{\gamma(\alpha,\beta)})$); this will be one of our concerns in the sequel.

The content Proposition~\ref{minimalprop} is essentially the only property of minimal walks we need. A discussion of more sophisticated applications is beyond the scope of this paper.We refer the reader to~\cite{acta} or~\cite{stevobook} for more information.

We will, however, need one a generalization of the minimal walks machinery in order to handle some issues that arise when dealing with successors of singular cardinals of countable cofinality.  These techniques were introduced in~\cite{819}, and used again in~\cite{ideals}.

\begin{definition}
\label{generalizeddef}
Let $\lambda$ be a cardinal. A {\em generalized $C$-sequence} is a family
\begin{equation*}
\bar{e}=\langle e^m_\alpha:\alpha<\lambda, m<\omega\rangle
\end{equation*}
such that for each $\alpha<\lambda$ and $m<\omega$,
\begin{itemize}
\item $e^m_\alpha$ is closed unbounded in $\alpha$, and
\sk
\item $e^m_\alpha\subs e^{m+1}_\alpha$.
\sk
\end{itemize}
\end{definition}

One can think of a generalized $C$-sequence as a countable family of $C$-sequences which are increasing in a sense. One can also utilize generalized $C$-sequences in the context of minimal walks. In this paper, we do this in the simplest fashion --- given $m<\omega$ and $\alpha<\beta<\lambda$, we let the {\em $m$-walk from $\beta$ to $\alpha$ along $\bar{e}$} consist of the minimal walk from $\beta$ to $\alpha$ using the $C$-sequence $\langle e^m_\gamma:\gamma<\lambda\rangle$. Such walks have their associated parameters, and we use the superscript $m$ to indicate which part of the generalized $C$-sequence is being used in computations. So, for example, the $m$-walk from $\beta$ to $\alpha$ along $\bar{e}$ will have length $\rho_2^m(\alpha,\beta)$, and consist of ordinals denoted $\beta^m_i(\alpha,\beta)$ for $i\leq\rho^m_2(\alpha,\beta)$.

The requirement that $e^m_\alpha\subs e^{m+1}_\alpha$ is relevant for the following reason.  Given $\alpha<\beta$, we note that the sequence $\langle \min(e_\beta^m\setminus\alpha):m<\omega\rangle$ is non-increasing, and therefore eventually constant. From this it follows easily that the $m$-walk from $\beta$ to~$\alpha$  along $\bar{e}$ is exactly the same for all sufficiently large $m$.

\medskip

\noindent{\sf Club-guessing}

\medskip

Our arguments are going to make use of generalized $C$-sequences that have been carefully selected to interact with certain club-guessing sequences.  The type of club-guessing sequence we use depends on whether or not the cofinality of our singular cardinal~$\mu$ is uncountable, so we deal with each case separately. In both cases, we will be defining a stationary set $S$, a club-guessing sequence $\bar{C}$, and a generalized $C$-sequence $\bar{e}$.

If $\cf(\mu)>\aleph_0$, then we define
\begin{equation}
S:= S^{\mu^+}_{\cf(\mu)}=\{\delta<\mu^+:\cf(\delta)=\cf(\mu)\}.
\end{equation}
By Claim~2.6 on page~127 of~\cite{cardarith} (or see Theorem~2 of \cite{819}), there is a sequence $\langle C_\delta:\delta\in S\rangle$ such that
\begin{itemize}
\sk
\item $C_\delta$ is club in $\delta$,
\sk
\item $\otp(C_\delta)=\cf(\mu)$,
\sk
\item $\langle\cf(\alpha):\alpha\in\nacc(C_\delta)\rangle$ increases to $\mu$, and
\sk
\item whenever $E$ is club in~$\mu^+$, there are stationarily many $\delta\in S$ for which $C_\delta\subs E$.
\sk
\end{itemize}
Here ``$\nacc(C_\delta)$'' refers to the non-accumulation points of $C_\delta$, that is, those elements of $C_\delta$ that are not limits of points in $C_\delta$.

We now use the ``ladder swallowing'' trick (see Lemma~13 of~\cite{nsbpr}) to build a $C$-sequence $\langle e_\alpha:\alpha<\mu^+\rangle$ such that for each $\alpha<\mu^+$,
\begin{gather}
|e_\alpha|< \mu,\\
\intertext{and}
\delta\in S\cap e_\alpha\Longrightarrow C_\delta\subs e_\alpha.
\end{gather}
We then construct a ``silly'' generalized $C$-sequence $\bar{e}=\langle e^m_\alpha:m<\omega,\alpha<\mu^+\rangle$ by setting $e^m_\alpha = e_\alpha$ for all $m<\omega$.

In the case where $\mu$ is of countable cofinality, our definition of $S$, $\bar{C}$, and $\bar{e}$ is a little more involved as it is an open question whether one can find club-guessing sequences analogous to those above. Our argument will rely on technology developed in~\cite{ideals}.

In this case, we start by setting
\begin{equation}
S:= S^{\mu^+}_{\aleph_1}=\{\delta<\mu^+:\cf(\delta)=\aleph_1\},
\end{equation}
and assume $\langle \mu_i:i<\omega\rangle$ is an increasing sequence of uncountable cardinals cofinal in~$\mu$.

We are going to present a simplified version of the conclusion of Theorem~4 of~\cite{ideals}; the reader can consult that paper for a detailed proof (Proposition~5.8 is particularly relevant).

Thus, there is a sequence $\langle C_\delta:\delta\in S\rangle$ such that each $C_\delta$ is club in~$\delta$, and $C_\delta=\bigcup_{m<\omega}C_\delta[m]$ where
\begin{itemize}
\sk
\item $C_\delta[m]$ is club in $\delta$,
\sk
\item $|C_\delta[m]|\leq\mu_m^+$,
\sk
\end{itemize}
and such that for every club $E\subs\mu^+$, there are stationarily many $\delta\in S$ such that for each $m<\omega$, $\nacc(C_\delta[m])\cap E$ contains unboundedly many ordinals of cofinality greater than $\mu_m^+$.  (Note that the use of ``$\nacc$'' is redundant as the cofinality assumption guarantees such an ordinal cannot be a limit point of $C_\delta[m]$.)

Lemma~5.10 of~\cite{ideals} provides a generalized $C$-sequence $\bar{e}=\langle e_\alpha^m:\alpha<\mu^+,m<\omega\rangle$ such that
\begin{gather}
\label{2.14}|e^m_\alpha|\leq\cf(\alpha)+\mu_m^+\\
\intertext{and}
\label{2.15}\delta\in S\cap e_\alpha^m\Longrightarrow C_\delta[m]\subs e_\alpha^m.
\end{gather}

In either case, the phrase ``choose $\delta\in S$ such that $C_\delta$ guesses $E$'' should be given the obvious meaning.

\medskip

\noindent{\sf Scales}

\medskip

The next ingredient we need for our theorem is the concept of a scale for a singular cardinal.

\begin{definition}
Let $\mu$ be a singular cardinal. A {\em scale for
$\mu$} is a pair $(\vec{\mu},\vec{f})$ satisfying
\begin{enumerate}
\item $\vec{\mu}=\langle\mu_i:i<\cf(\mu)\rangle$ is an increasing sequence of regular cardinals
such that $\sup_{i<\cf(\mu)}\mu_i=\mu$ and $\cf(\mu)<\mu_0$.
\item $\vec{f}=\langle f_\alpha:\alpha<\mu^+\rangle$ is a sequence of functions such that
\begin{enumerate}
\item $f_\alpha\in\prod_{i<\cf(\mu)}\mu_i$.
\item If $\gamma<\delta<\beta$ then $f_\gamma<^* f_\beta$, where  the notation $f<^* g$  means that $\{i<\cf(\mu): g(i)\leq f(i)\}$ is bounded in $\cf(\mu)$.
\item If $f\in\prod_{i<\cf(\mu)}\mu_i$ then there is an $\alpha<\beta$ such that $f<^* f_\alpha$.
\end{enumerate}
\end{enumerate}
\end{definition}

It is an important theorem of Shelah~(see page Main Claim~1.3 on page~46 of~\cite{cardarith}) that scales exist for any singular $\mu$; readers seeking a gentler exposition of this and related topics can consult~\cite{cummings}, or~\cite{myhandbook}. If $\mu$ is singular and $(\vec{\mu},\vec{f})$ is a scale for $\mu$, then there is a
natural way to color the pairs of ordinals $\alpha<\beta<\mu^+$ using $\cf(\mu)$ colors, namely
\begin{equation}
\label{deltadef}
\Gamma(\alpha,\beta)=
\sup(\{i<\cf(\mu):f_\beta(i)\leq f_\alpha(i)\})
\end{equation}

The coloring $\Gamma$ is the critical ingredient in Shelah's proof of $\mu^+\nrightarrow[\mu^+]^2_{\cf(\mu)}$ for
singular~$\mu$, and it plays a central role in the sequel as well. One can consult Conclusion~4.1(a) on page~67 of~\cite{cardarith})or Section~5 of~\cite{myhandbook} (among many other places) for an exposition of this.

We need one standard fact about scales in our proof.  We remind the reader that notation of the form ``$(\exists^*\beta<\lambda)\psi(\beta)$''
  means $\{\beta<\lambda:\psi(\beta)\text{ holds}\}$ is unbounded below~$\lambda$, while
    ``$(\forall^*\beta<\lambda)\psi(\beta)$'' means that $\{\beta<\lambda:\psi(\beta)\text{ fails}\}$ is bounded
     below~$\lambda$.

\begin{lemma}
\label{scalelemma}
Let $(\vec{\mu},\vec{f})$ be a scale for $\mu$. Then
\begin{equation}
(\forall^*i<\cf(\mu))(\forall \eta<\mu_i)(\exists^*\alpha<\mu^+)[\eta<f_\alpha(i)].
\end{equation}
\end{lemma}
\begin{proof}
If not, one easily obtains a contradiction to $(\vec{\mu},\vec{f})$ being a scale.
\end{proof}

\medskip

\noindent{\sf Elementary Submodels}

\medskip

We have the usual conventions when dealing with elementary submodels. In brief, we always assume that $\chi$ is regular cardinal much larger than anything relevant to our theorem, and we  let $\mathfrak{A}$ denote the structure $\langle H(\chi),\in, <_\chi\rangle$ where $H(\chi)$ is the collection of sets hereditarily of cardinality less than $\chi$, and $<_\chi$ is some well-order of $H(\chi)$. The use of $<_\chi$ means that our structure $\mathfrak{A}$ has definable Skolem functions and it makes sense to talk about Skolem hulls. In general, if $B\subs H(\chi)$, then we denote the Skolem hull of $B$ in $\mathfrak{A}$ by $\Sk_{\mathfrak{A}}(B)$.

The following technical lemma due originally to Baumgartner~\cite{jb} (see the last section of~\cite{myhandbook}, or~\cite{nsbpr} for a proof) is crucial for our work.

\begin{lemma}
\label{newcharlem}
Assume that $M\prec\mathfrak{A}$ and let $\sigma\in M$ be a cardinal.  If we define $N=\Sk_{\mathfrak{A}}(M\cup\sigma)$
then for all regular cardinals $\tau\in M$ greater than $\sigma$, we have
\begin{equation*}
\sup(M\cap\tau)=\sup(N\cap\tau).
\end{equation*}
\end{lemma}

As a corollary to the above, we can deduce an important fact about {\em characteristic functions} of models,
which we define next.

\begin{definition}
\label{chardef}
Let $\mu$ be a singular cardinal of cofinality~$\kappa$, and let $\vec{\mu}=\langle\mu_i:i<\kappa\rangle$ be an increasing
sequence of regular cardinals cofinal in $\mu$.  If $M$ is an elementary submodel of $\mathfrak{A}$
such that
\begin{itemize}
\item $|M|<\mu$,
\item $\langle \mu_i:i<\cf(\mu)\rangle\in M$, and
\item $\kappa+1\subs M$,
\end{itemize}
then the {\em characteristic function of $M$ on $\vec{\mu}$} (denoted $\Ch^{\vec{\mu}}_M$) is the function
with domain $\kappa$ defined by
\begin{equation*}
\Ch^{\vec{\mu}}_M(i):=
\begin{cases}
\sup(M\cap\mu_i) &\text{if $\sup(M\cap\mu_i)<\mu_i$,}\\
0  &\text{otherwise.}
\end{cases}
\end{equation*}
If $\vec{\mu}$ is clear from context, then we suppress reference to it in the notation.
\end{definition}

In the situation of Definition~\ref{chardef}, it is clear that $\Ch^{\vec{\mu}}_M$ is an element of the product
 $\prod_{i<\kappa}\mu_i$, and furthermore, $\Ch^{\vec{\mu}}_M(i)=\sup(M\cap\mu_i)$ for all
sufficiently large $i<\kappa$.  We can now see that the following corollary follows immediately from Lemma~\ref{newcharlem}.

\begin{corollary}
\label{skolemhulllemma}
Let $\mu$, $\kappa$, $\vec{\mu}$, and $M$ be as in Definition~\ref{chardef}.
If $i^*<\kappa$ and we define $N$ to be $\Sk_{\mathfrak{A}}(M\cup\mu_{i^*})$,
then
\begin{equation}
\Ch_M\restr [i^*+1,\kappa)=\Ch_N\restr [i^*+1,\kappa).
\end{equation}
\end{corollary}

We introduce one more bit of notation concerning elementary submodels.

\begin{definition}
Let $\lambda$ be a regular cardinal. A $\lambda$-approximating sequence is a continuous $\in$-chain
$\mathfrak{M}=\langle M_i:i<\lambda\rangle$ of elementary submodels of $\mathfrak{A}$ such that
\begin{enumerate}
\item $\lambda\in M_0$,
\item $|M_i|<\lambda$,
\item $\langle M_j:j\leq i\rangle\in M_{i+1}$, and
\item $M_i\cap\lambda$ is a proper initial segment of $\lambda$.
\end{enumerate}
If $x\in H(\chi)$, then we say that $\mathfrak{M}$ is a $\lambda$-approximating sequence over $x$ if
$x\in M_0$.
\end{definition}

Note that if $\mathfrak{M}$ is a $\lambda$-approximating sequence and $\lambda=\mu^+$, then $\mu+1\subs M_0$ because
of condition~(4) and the fact that $\mu$ is an element of each $M_i$.

\section{Main Lemma}

In this section we prove a lemma which shows that the generalized $C$-sequences isolated in the preceding section have some very nice properties. The following {\em ad hoc} definition is key; note that the terminology implicitly assumes  the presence of a generalized $C$-sequence in the background.

\begin{definition}
\label{varphidefn}
Suppose $k$ and $m$ are natural numbers, and $\eta<\mu^+$.  The formula $\varphi_{k,m,\eta}(\beta^*,\beta)$ says
\begin{enumerate}
\item $\beta^*<\beta$,
\sk
\item $\rho^m_2(\beta^*,\beta)=k$,
\sk
\item $\eta=\eta^m(\beta^*,\beta)$, and
\sk
\item $\eta<\beta^*$.
\end{enumerate}
\end{definition}

The formula $\varphi_{k, m,\eta}$ isolates a particular configuration of ordinals, a configuration whose importance can be glimpsed in the following lemma:

 \begin{lemma}
 \label{varphilemma}
 If $\varphi_{k, m,\eta}(\beta^*,\beta)$ holds, then
\begin{equation}
\label{point1}
\eta<\alpha\leq\beta^*\Longrightarrow \beta^m_i(\alpha,\beta)=\beta^m_i(\beta^*,\beta)\text{ for }i\leq k.
\end{equation}
Given the role of $k$, we see that if $\varphi_{m, k,\eta}(\beta^*,\beta)$ holds, then
\begin{equation}
\label{point2}
\eta<\alpha\leq\beta^*\Longrightarrow \beta^m_k(\alpha,\beta)=\beta^*.
\end{equation}
\end{lemma}
\begin{proof}
This follows immediately from Proposition~\ref{minimalprop}.
\end{proof}

We come now to the main  lemma of this paper:

\begin{lemma}
\label{mainlemma}
Let $\mu$ be a singular cardinal, and let $\bar{e}$ be a generalized $C$-sequence as in the preceding section. Then for any unbounded $A\subs\mu^+$, there are $k$, $m$, and $\eta$ for which
\begin{equation}
(\exists^{\stat}\beta^*<\mu^+)(\exists^*\beta\in A)[\varphi_{k, m, \eta}(\beta^*,\beta)].
\end{equation}
\end{lemma}
\begin{proof}

Let $S$, $\bar{C}$, and $\bar{e}$ be as in previous section's discussion, and let $A\subs\mu^+$ be unbounded.  We set
\begin{equation}
x:=\{\mu, S, \bar{C},\bar{e},A\}
\end{equation}
and let $\langle M_\alpha:\alpha<\mu^+\rangle$ be a $\mu^+$-approximating sequence over~$x$.  Define
\begin{equation}
E:=\{\delta<\mu^+:M_\delta\cap\mu^+=\delta\},
\end{equation}
and fix $\delta\in S$ such that $C_\delta$ guesses $E$ in the appropriate sense. In order to find $k$, $m$, and $\eta$ we must divide into cases.

\medskip
{\sf Case $\cf(\mu)>\aleph_0$:}
\medskip

In this situation, we set $m=1$ (recall that $\bar{e}$ is ``silly'' in this case) and $k=\rho^1_2(\delta,\beta)$. Next, fix $\beta^*$ such that
\begin{gather}
\beta^*\in\nacc(C_\delta)\cap E,\\
\gamma^{1,-}(\delta,\beta)<\max(C_\delta\cap\beta^*),\\
\intertext{and}
\label{coflarge}\cf(\beta^*)>|e^1_{\gamma^1(\delta,\beta)}|.
\end{gather}
Notice that these conditions are satisfied by all sufficiently large $\beta^*\in\nacc(C_\delta)$ because of our assumptions on $\bar{C}$ and $\bar{e}$. We now define
\begin{equation}
\eta:=\sup(e^1_{\beta^1_{k-1}(\beta^*,\beta)}\cap\beta^*),
\end{equation}
and we claim $\varphi_{k,m,\eta}(\beta^*,\beta)$ holds.

Clearly $\beta^*<\beta$, so the first requirement is of no concern.  Since
\begin{equation}
\gamma^{1,-}(\delta,\beta)<\beta^*<\delta,
\end{equation}
we know that
\begin{gather}
\beta^1_i(\beta^*,\beta)=\beta^1_i(\delta,\beta),\\
\intertext{and}
\beta^{1,-}_i(\beta^*,\beta)=\beta^{1,-}_i(\delta,\beta)
\end{gather}
for all $i<\rho^1_2(\delta,\beta)$.

By definition, we have
\begin{equation}
\delta\in e^1_{\gamma^1(\delta,\beta)}=e^1_{\beta^1_{k-1}(\delta,\beta)},
\end{equation}
and so by our choice of $\bar{e}$ we obtain
\begin{equation}
\label{barechoice}
\beta^*\in C_\delta\subs e^1_{\gamma^1(\delta,\beta)}=e^1_{\beta^1_{k-1}(\delta,\beta)}=e^1_{\beta^1_{k-1}(\beta^*,\beta)}.
\end{equation}
It follows immediately that
\begin{equation}
\rho^1_2(\beta^*,\beta)=k,
\end{equation}
and so we have obtained the second requirement of Definition~\ref{varphidefn}.

Next, we note that for $i<k$ we have
\begin{align*}
\beta^{1,-}_i(\beta^*,\beta)=\beta^{1,-}_i(\delta,\beta)&\leq\gamma^{1,-}(\delta,\beta)\\
&<\max(C_\delta\cap\beta^*)\leq\sup(e^1_{\beta^1_{k-1}(\beta^*,\beta)}\cap\beta^*)=\eta.
\end{align*}
From this, we see
\begin{equation}
\eta=\eta^1(\beta^*,\beta),
\end{equation}
and we have met the third demand of Definition~\ref{varphidefn}.

Finally, our requirement~(\ref{coflarge}) taken together with~(\ref{barechoice}) lets us conclude
\begin{equation}
\beta^*\in\nacc(e^1_{\beta^1_{k-1}(\beta^*,\beta)}),
\end{equation}
and therefore
\begin{equation}
\eta=\sup(e^1_{\beta^1_{k-1}(\beta^*,\beta)}\cap\beta^*)<\beta^*.
\end{equation}

\medskip
\noindent{\sf Case $\cf(\mu)=\aleph_0$:}
\medskip

In this situation we must work a little harder.  First, we define
\begin{equation}
\gamma^*:=\sup\{\gamma^{m,-}(\delta,\beta):m<\omega\}.
\end{equation}
Since $\cf(\delta)>\aleph_0$, we know that $\gamma^*<\delta$, and
\begin{equation}
\beta^{m,-}_i(\delta,\beta)\leq\gamma^*\text{ for all }m<\omega\text{ and }i<\rho^m_2(\delta,\beta).
\end{equation}

Next (see the discussion after Definition~\ref{generalizeddef}) we fix $m^*<\omega$ so that
\begin{equation}
m^*\leq m<\omega\Rightarrow\langle \beta^m_i(\delta,\beta):i<\rho^m_2(\delta,\beta)\rangle=\langle \beta^{m^*}_i(\delta,\beta):i<\rho^{m^*}_2(\delta,\beta)\rangle,
\end{equation}
and define
\begin{equation}
k:=\rho^{m^*}_2(\delta,\beta).
\end{equation}

We then let $m\geq m^*$ be the least natural number for which
\begin{equation}
\cf(\gamma^{m^*}(\delta,\beta))\leq\mu^+_m.
\end{equation}

Taking this together with (\ref{2.14}), we conclude
\begin{equation}
|e^{m}_{\gamma^m(\delta,\beta)}|=|e^{m}_{\gamma^{m^*}(\delta,\beta)}|\leq \mu^+_m.
\end{equation}
Note as well that~(\ref{2.15}) tells us
\begin{equation}
\label{3.25}
C_\delta[m]\subs e^m_{\gamma^m(\delta,\beta)}
\end{equation}
as well.

Our assumptions on $\bar{C}$ now allow us to find $\beta^*$ satisfying the following:
\begin{itemize}
\sk
\item $\beta^*\in\nacc(C_\delta[m])\cap E$,
\sk
\item $\cf(\beta^*)>\mu^+_m$, and
\sk
\item $\gamma^*<\max(C_\delta[m]\cap\beta^*)$.
\end{itemize}
Note that the last requirement can be achieved because the set of candidates satisfying the first two demands is unbounded in $\delta$.

We now define
\begin{equation}
\eta:=\sup(e^m_{\gamma^m(\delta,\beta)}\cap\beta^*).
\end{equation}

The verification that $\varphi_{k, m,\eta}(\beta^*,\beta)$ holds follows the same broad outline as we saw in the preceding case. Once again, since $\beta^*\in C_\delta$ it is immediate that $\beta^*<\beta$.

Since $\gamma^*<\beta^*<\delta$ and $m\geq m^*$, it follows that for $i<k$ we have
\begin{equation}
\beta^m_i(\beta^*,\beta)=\beta^m_i(\delta,\beta)=\beta^{m^*}_i(\delta,\beta),
\end{equation}
and in particular
\begin{equation}
\beta^m_{k-1}(\beta^*,\beta)=\gamma^m(\delta,\beta).
\end{equation}
By~(\ref{3.25}), it follows that
\begin{equation}
\beta^m_k(\beta^*,\beta)=\min(e^m_{\beta^m_{k-1}(\beta^*,\beta)}\setminus\beta^*)=\beta^*
\end{equation}
and we conclude
\begin{equation}
\rho^m_2(\beta^*,\beta)=k.
\end{equation}

The fact that $\eta=\eta^m(\beta^*,\beta)$ also follows easily as we have ensured
\begin{equation}
\gamma^*<\max(C_\delta[m]\cap\beta^*)\leq\sup(e^m_{\gamma^m(\beta^*,\beta)}\cap\beta^*)=\sup(e^m_{\gamma^m(\delta,\beta)}\cap\beta^*)=\eta.
\end{equation}
Finally, since
\begin{equation}
\cf(\beta^*)>\mu^+_m\geq |e^m_{\gamma^m(\delta,\beta)}|=|e^m_{\beta^m_{k-1}(\delta,\beta)}|,
\end{equation}
it follows that
\begin{equation}
\eta:=\sup(e^m_{\gamma^m(\delta,\beta)}\cap\beta^*)<\beta^*,
\end{equation}
and so $\varphi_{k,m,\eta}(\beta^*,\beta)$ holds.

Combining the two cases, we find that we have $k$, $m$, $\eta$, $\beta^*$, $\delta$, and $\beta$ such that
\begin{itemize}
\sk
\item $\varphi_{k,m,\eta}(\beta^*,\beta)$ holds,
\sk
\item $\beta^*<\delta<\beta$, and
\sk
\item both $\beta^*$ and $\delta$ are in $E$.
\sk
\end{itemize}

We finish the proof using standard elementary submodel arguments.  Since the model $M_\delta$ contains $x$ together with $k$, $m$, $\eta$, and $\beta^*$, but $M_\delta\cap\mu^+=\delta\leq\beta$, it follows that
\begin{equation}
(\exists^*\beta\in A)[\varphi_{k, m,\eta}(\beta^*,\beta)].
\end{equation}
Since $x$, $k$, $m$, and $\eta$ are in $M_{\beta^*}$ and $\beta^*=M_{\beta^*}\cap\mu^+$, we obtain
\begin{equation}
(\exists^{\stat}\beta^*<\mu^+)(\exists^*\beta\in A)[\varphi_{k,m,\eta}(\beta^*,\beta)],
\end{equation}
as required.
\end{proof}
\section{Main Theorem}

\begin{theorem}[Main Theorem]
\label{theorem2}
Suppose $\mu$ is a singular cardinal. There is a function
 \begin{equation}
 D:[\mu^+]^2\rightarrow [\mu^+]^2\times\cf(\mu)
 \end{equation}
 such that for any unbounded $A\subs\mu^+$, there is a stationary $S\subs\mu^+$ such that
 \begin{equation}
 [S]^2\times\cf(\mu)\subs\ran(D\restr [A]^2).
 \end{equation}
\end{theorem}
\begin{proof}
Let $(\vec{\mu},\vec{f})$ be a scale for~$\mu$, and let $\bar{e}$ be a generalized $C$-sequence as in Lemma~\ref{mainlemma}. Choose a function
\begin{equation}
\iota:\cf(\mu)\rightarrow\omega\times\omega\times\cf(\mu)
\end{equation}
such that for any natural numbers $k$ and $m$, and any $\delta<\cf(\mu)$, there are unboundedly many $\gamma<\cf(\mu)$ such that $\iota(\gamma)=\langle k, m, \delta\rangle$. Let $\Gamma:[\mu^+]^2\rightarrow\cf(\mu)$ be the function from~(\ref{deltadef}).

The definition of $D$ will require several other auxiliary functions defined on $[\mu^+]^2$.  First, we let $k$, $m$, and $\delta$ be the two-place functions defined by the recipe
\begin{equation}
\iota(\Gamma(\alpha,\beta))=\langle k(\alpha,\beta), m(\alpha,\beta),\delta(\alpha,\beta)\rangle.
\end{equation}
We then define
\begin{gather}
\beta^*(\alpha,\beta)=\beta^{m(\alpha,\beta)}_{k(\alpha,\beta)}(\alpha,\beta),\\
\eta^*(\alpha,\beta)=\eta^{m(\alpha,\beta)}(\beta^*(\alpha,\beta),\beta),\\
\alpha^*(\alpha,\beta)=\beta^{m(\alpha,\beta)}_{k(\alpha,\beta)}(\eta^*(\alpha,\beta)+1,\alpha),\\
\intertext{and}
D(\alpha,\beta)= \langle\{\alpha^*(\alpha,\beta),\beta^*(\alpha,\beta)\},\delta(\alpha,\beta)\rangle.
\end{gather}

The computation of $D(\alpha,\beta)$ can be described in English as follows.  Given $\alpha<\beta$, we use $\iota$ and $\Gamma$ to obtain $k$, $m$, and $\delta$.  The ordinal $\beta^*$ is the $k$-th step in the $m$-walk from $\beta$ to $\alpha$, and $\eta^*$ is the corresponding value of $\eta^m(\beta^*,\beta)$ computed from this walk.  The ordinal~$\alpha^*$ is then the $k$-th step in the $m$-walk from $\alpha$ down to $\eta^*+1$, and $D$ returns the value $\langle\{\alpha^*,\beta^*\},\delta\rangle$.

Given an unbounded $A\subs\mu^+$, we fix $k$, $m$, and $\eta$ as in Lemma~\ref{mainlemma}. Now  define
\begin{equation}
S^*:=\{\beta^*<\mu^+:(\exists^*\beta\in A)\varphi_{k,m,\eta}(\beta^*,\beta)\}
\end{equation}
and
\begin{equation}
x:=\{\bar{e},(\vec{\mu},\vec{f}),A,\eta\}.
\end{equation}

Let $\langle M_\alpha:\alpha<\mu^+\rangle$ be a $\mu^+$-approximating sequence over $x$, and define
\begin{equation}
S:=\{\delta\in S:M_\delta\cap \mu^+=\delta\}.
\end{equation}

We claim that the stationary set $S$ satisfies the conclusion of the theorem. Thus,
given $\alpha^*<\beta^*$ in $S$ and $\delta<\cf(\mu)$, we must find $\alpha<\beta$ in $A$ such that $D(\alpha,\beta)=\langle \{\alpha^*,\beta^*\},\delta\rangle$. We will do this by striving for the following goal:

\medskip

\noindent{\sf Goal:} Find $\alpha<\beta$ in $A$ such that
\begin{enumerate}
\item $\alpha^*<\alpha<\beta^*<\beta$,
\sk
\item $\varphi_{k,m,\eta}(\beta^*,\beta)$,
\sk
\item $\varphi_{k,m,\eta}(\alpha^*,\alpha)$, and
\sk
\item $\iota(\Gamma(\alpha,\beta))=\langle k,m,\delta\rangle$.
\end{enumerate}

\begin{proposition}
If $\alpha$ and $\beta$ are as above, then $D(\alpha,\beta)=\langle\{\alpha^*,\beta^*\},\delta\rangle$.
\end{proposition}
\begin{proof}
By~(4), we know
\begin{gather}
k(\alpha,\beta)=k,\\
m(\alpha,\beta)=m,\\
\intertext{and}
\delta(\alpha,\beta)=\delta.
\end{gather}
Since $\varphi_{k,m,\eta}(\beta^*,\beta)$ holds and $\eta<\alpha<\beta^*$, an application of Lemma~\ref{varphilemma} tells us
\begin{equation}
\beta^*(\alpha,\beta)=\beta^m_k(\alpha,\beta)=\beta^*.
\end{equation}

The definition of $\eta^*$ together with the fact that $\varphi_{k,m,\eta}(\beta^*,\beta)$ holds informs us that
\begin{equation}
\eta^*(\alpha,\beta)=\eta^m(\beta^*,\beta)=\eta,
\end{equation}
and we can once again apply Lemma~\ref{varphilemma} (this time using $\varphi_{k,m,\eta}(\alpha^*,\alpha)$) to conclude that
\begin{equation}
\alpha^*(\alpha,\beta)=\beta^m_k(\eta+1,\alpha)=\alpha^*,
\end{equation}
as required.

\end{proof}

So how do we go about obtaining our goal?
We start by choosing $\beta\in A$ for which $\varphi_{k, m,\eta}(\beta^*,\beta)$ is true.  We set
\begin{equation}
y:=x\cup\{\alpha^*\},
\end{equation}
and define
\begin{equation}
\mathcal{M}:=\Sk_{\mathfrak{A}}(y\cup\cf(\mu)).
\end{equation}
Since $y\in M_{\alpha^*+1}$, the construction of $\mathcal{M}$ can be done in the model $M_{\alpha^*+2}$ by taking the Skolem hull of $y\cup\cf(\mu)$ in the model $M_{\alpha^*+1}$. Thus,
\begin{equation}
\mathcal{M}\in M_{\alpha^*+1}\subs M_{\beta^*}.
\end{equation}
From this it follows that
\begin{equation*}
\Ch_{\mathcal{M}}^{\vec{\mu}}\in M_{\beta^*}\cap\prod_{i<\cf(\mu)}\mu_i.
\end{equation*}
Since $\beta^*=M_{\beta^*}\cap\mu^+<\beta$ and $(\vec{\mu},\vec{f})$ is a scale,  we conclude that there is an $i_0<\cf(\mu)$ such that
\begin{equation}
\Ch_{\mathcal{M}}^{\vec{\mu}}(i)<f_\beta(i)\text{ whenever }i_0\leq i<\cf(\delta).
\end{equation}

Our next move is to note that since $I:=\{\alpha\in A:\varphi_{k, m,\eta}(\alpha^*,\alpha)\}$ is unbounded, the sequence $\langle f_\alpha:\alpha\in I\rangle$ together with $\vec{\mu}$ forms (modulo re-indexing) a scale for $\mu$. Thus we can apply Lemma~\ref{scalelemma} and fix a value $i_1$ such that whenever $i_1\leq i<\cf(\delta)$,
\begin{equation}
(\forall\zeta<\mu_i)(\exists^*\alpha\in A)[\varphi_{k, m,\eta}(\alpha^*,\alpha)\wedge \zeta<f_\alpha(i)].
\end{equation}
Given our choice of the function $\iota$, it follows that we can choose $i^*<\cf(\mu)$ such that
 $\max\{i_0, i_1\}<i^*$ and $\iota(i^*)=\langle k,m,\delta\rangle$.  In particular, for this choice of~$i^*$ we have
 \begin{gather}
 \label{condition1}
 \Ch^{\vec{\mu}}_{\mathcal{M}}\restr [i^*,\cf(\mu))<f_\beta\restr[i^*,\cf(\mu)),\\
 \label{condition2}
 (\forall\zeta<\mu_{i^*})(\exists^*\alpha\in A)[\varphi_{k_0, m_0}(\eta_0,\alpha^*,\alpha)\wedge \zeta<f_\alpha(i^*)],\\
 \intertext{and}
 \iota(i^*)=\langle k,m, \delta\rangle.
 \end{gather}

 We now define
 \begin{equation}
 \mathcal{N}=\Sk_{\mathfrak{A}}(\mathcal{M}\cup\mu_{i^*}).
 \end{equation}

 Notice that~(\ref{condition2}) holds in $\mathcal{N}$ as this model contains $i^*$ and all parameters relevant to the formula. We have also ensured that the ordinal $f_\beta(i^*)$ is in $\mathcal{N}$ too. Thus, we can find an ordinal $\alpha$ such that
 \begin{gather}
 \alpha\in\mathcal{N}\cap A,\\
 \varphi_{k_0, m_0}(\eta_0,\alpha^*,\alpha),\\
 \intertext{and}
 f_\beta(i^*)<f_\alpha(i^*).
 \end{gather}

We claim now that $\alpha$ and $\beta$ are as required. The following statements are immediate from our preceding work:
\begin{itemize}
\sk
\item $\alpha$ and $\beta$ are in $B$,
\sk
\item $\alpha^*<\alpha<\beta^*<\beta$,
\sk
\item $\varphi_{k, m,\eta}(\alpha^*,\alpha)$, and
\sk
\item $\varphi_{k, m,\eta}(\beta^*,\beta)$,
\sk
\end{itemize}
and so we will achieve our goal provided we can show $\Gamma(\alpha,\beta)=i^*$.

This, however, follows almost immediately by a standard argument.
Since $f_\beta(i^*)<f_\alpha(i^*)$, it is clear that $i^*\leq\Gamma(\alpha,\beta)$.  By Lemma~\ref{skolemhulllemma}, we know
\begin{equation}
\Ch^{\vec{\mu}}_{\mathcal{N}}\restr[i^*+1,\cf(\delta))=\Ch^{\vec{\mu}}_{\mathcal{M}}\restr[i^*+1,\cf(\delta)),
\end{equation}
and so~(\ref{condition1}) implies $\Gamma(\alpha,\beta)\leq i^*$ as well. Thus $\Gamma(\alpha,\beta)=i^*$, and we have achieved our goal. As noted before, this is enough to finish the proof of the theorem.

\end{proof}

\section{Conclusions}

In this last section we will deduce several results as corollaries of Theorem~\ref{theorem2}, including those results mentioned in our introduction.

\begin{proposition}
\label{cor1}
Suppose $\mu$ is a singular cardinal, and let $\langle\theta_i:i<\cf(\mu)\rangle$ be a sequence
of cardinals with supremum~$\theta^*$.  If $\mu^+\nrightarrow[\mu^+]^2_{\theta_i}$ for each $i$, then
\begin{equation}
\mu^+\nrightarrow [\mu^+]^2_{\theta^*}.
\end{equation}
\end{proposition}
\begin{proof}
For each $\delta<\cf(\mu)$, let $c_\delta:[\mu^+]^2\rightarrow\theta_\delta$ witness $\mu^+\nrightarrow[\mu^+]^2_{\theta_\delta}$. Using the notation of the preceding section, we define a coloring
$d:[\mu^+]^2\rightarrow\theta^*$ by
\begin{equation}
d(\alpha,\beta)=c_{\delta(\alpha,\beta)}(\alpha^*(\alpha,\beta),\beta^*(\alpha,\beta)).
\end{equation}

Given an unbounded $A\subs\mu^+$, let $S$ be the stationary set provided by Theorem~\ref{theorem2}. Given $\varsigma<\theta^*$, we choose $\delta$ such that $\varsigma<\theta_\delta$, and then select $\alpha^*<\beta^*$ in $S$ for which
\begin{equation}
c_{\delta}(\alpha^*,\beta^*)=\varsigma.
\end{equation}
By Theorem~\ref{theorem2}, there are $\alpha<\beta$ in $A$ for which
\begin{equation}
D(\alpha,\beta)=\langle \{\alpha^*,\beta^*\},\delta\rangle,
\end{equation}
and so
\begin{equation}
d(\alpha,\beta)=c_{\delta(\alpha,\beta)}(\alpha^*(\alpha,\beta),\beta^*(\alpha,\beta))=c_\delta(\alpha^*,\beta^*)=\varsigma,
\end{equation}
as required.
\end{proof}

Note that the preceding proof did not use the full strength of $\mu^+\nrightarrow[\mu^+]^2_{\theta_\delta}$, but see Proposition~\ref{finalprop} below.

\begin{corollary}
Let $\mu$ be a singular cardinal, and let $\theta$ be the least cardinal for which $\mu\rightarrow[\mu^+]^2_\theta$.
Then $\cf(\mu)<\cf(\theta)$.
\end{corollary}

The following proposition yields the result from the abstract as an immediate corollary.

\begin{proposition}
\label{maincor}
The following statements are equivalent for a singular cardinal~$\mu$:
\begin{enumerate}
\item $\mu^+\nrightarrow[\mu^+]^2_{\mu^+}$
\sk
\item $\mu^+\nrightarrow[\mu^+]^2_{\mu}$
\sk
\item $\mu^+\nrightarrow [\mu^+]^2_\theta$ for all $\theta<\mu$
\sk
\item $\mu^+\nrightarrow [\mu^+]^2_\theta$ for arbitrarily large $\theta<\mu$.
\sk
\end{enumerate}
\end{proposition}
\begin{proof}
The equivalence of (1) and (2) is given by Proposition~\ref{prop1}. Each statement on the list implies the next, we need
only establish that (4) implies (2), but this is immediate from Proposition~\ref{cor1}.

\end{proof}

Our final proposition also seems to be of interest.

\begin{proposition}
\label{finalprop}
Let $\mu$ be a singular cardinal, and let $\theta\leq\mu^+$. Then the following statements are equivalent:
\begin{enumerate}
\item $\mu^+\nrightarrow [\mu^+]^2_\theta$
\sk
\item There is a function $c:[\mu^+]^2\rightarrow\theta$ such that whenever $T\subs\mu^+$ is stationary and $\varsigma<\theta$,
there are $\alpha<\beta$ in $T$ with $c(\alpha,\beta)=\varsigma$.
\sk
\end{enumerate}
\end{proposition}
\begin{proof}
It is clear that (1) implies (2), so let us assume that $c$ is as in (2).  Define the function $d:[\mu^+]^2\rightarrow\theta$
by
\begin{equation}
d(\alpha,\beta)=c(\alpha^*(\alpha,\beta),\beta^*(\alpha,\beta)).
\end{equation}
Suppose now that we are given an unbounded $A\subs\mu^+$ and $\varsigma<\theta$.  Let $S$ be the stationary set guaranteed
to exist by Theorem~\ref{theorem2}, and choose $\alpha^*<\beta^*$ in $S$ with $c(\alpha^*,\beta^*)=\varsigma$.
Then there are $\alpha<\beta$ in $A$ for which $\alpha^*(\alpha,\beta)=\alpha^*$ and $\beta^*(\alpha,\beta)=\beta^*$, and
\begin{equation}
d(\alpha,\beta)= c(\alpha^*(\alpha,\beta),\beta^*(\alpha,\beta))=c(\alpha^*,\beta^*)=\varsigma
\end{equation}
as required.
\end{proof}

\end{document}